\documentclass[article]{amsart}
\usepackage{amssymb}
\usepackage{amsthm}
\usepackage{graphicx}
\numberwithin{equation}{section}
\newtheorem{theorem}{Theorem}[section]
\newtheorem{corollary}[theorem]{Corollary}
\newtheorem{proposition}[theorem]{Proposition}
\newtheorem{lemma}[theorem]{Lemma}
\newtheorem{definition}[theorem]{Definition}

\newtheorem{remark}[theorem]{Remark}
\newtheorem{example}{Example}

\begin{document}
\title[Tropical variants of some complex analysis results]{Tropical variants of some complex analysis results}

\author[K. Liu]{Kai Liu}

\address{Kai Liu \newline Department of Mathematics\\ Nanchang University\\ Nanchang, Jiangxi, 330031, P. R.
China} \email{liukai@ncu.edu.cn,}

\author[I. Laine]{Ilpo Laine}

\address{Ilpo Laine \newline Department of Physics and Mathematics\\ University of Eastern Finland\\ Joensuu, FI-80101, Finland} \email{ilpo.laine@uef.fi}

\author[K. Tohge]{Kazuya Tohge}
\address{Kazuya Tohge \newline Kanazawa University, College of Science and Engineering\\ Kanazawa 920-1192, Japan} \email{tohge@se.kanazawa-u.ac.jp}


\thanks{Financially supported by the NSFC No. 11301260, the NSF of Jiangxi No. 20132BAB211003, the YFED of Jiangxi of China No. GJJ13078, the China Scholarship Council No. 201406825034, the Academy of Finland No. 268009 and the JSPS Grant-in-Aid for Scientific Research (C) No. 25400131.}

\begin{abstract}
Tropical Nevanlinna theory studies value distribution of continuous piecewise linear functions of a real variable.  In this paper, we use the reasoning from tropical Nevanlinna theory to present tropical counterparts of some classical complex results related to Fermat type equations, Hayman conjecture and Br\"{u}ck conjecture.
\end{abstract}

\maketitle

\section{Introduction}

Tropical Nevanlinna theory may be understood as a cross-road between tropical mathematics and the classical Nevanlinna theory, see \cite{DB} for a general background of tropical mathematics and \cite{halburd}, \cite{KLT} and \cite{Laine1} for the tropical setting of the Nevanlinna theory. Indeed, tropical Nevanlinna theory provides the flexibility of applying complex analysis methods to considering real functions. Recall that Halburd and Southall \cite{halburd} described continuous piecewise linear functions on $\mathbf{R}$ with one-side integer derivatives as tropical meromorphic functions, and established tropical versions of the first main theorem and the lemma on the logarithmic derivative. Laine and Tohge \cite{Laine1} then showed that tropical Nevanlinna theory also holds to piecewise linear functions with arbitrary real slopes, and obtained a tropical version of the second main theorem.

Tropical Nevanlinna theory actually opens up possibilities for further investigations on value distribution and uniqueness theory of tropical meromorphic functions. In this paper, we shall consider tropical meromorphic solutions $y(x)$ to discrete equations of type
$$\sum_{j=0}^{q}n_{j}y(x+j)=1,$$
where the coefficients $n_{j}$ are integers and $q=1,2,3$. These considerations present, in some sense, tropical variants of some classical complex analysis topics such as Fermat type equations, a conjecture proposed to Hayman, and Br\"{u}ck conjecture type results. For the convenience of the reader, we shortly describe the respective complex analysis background before proceeding to the corresponding tropical reasoning.

Recall that we are considering, throughout, a max-plus semi-ring endowing $\mathbf{R}\cup\{-\infty\}$ with tropical addition $$x\oplus y:=\max(x,y)$$ and tropical multiplication $$x\otimes y:=x+y.$$ We also use $x\oslash y:=x-y$ and $x^{\otimes \alpha}=\alpha x$, for $\alpha\in\mathbf{R}$. The identity element $0_{\circ}$ for tropical addition is $0_{\circ}=-\infty$ and the identity element $1_{\circ}$ for multiplication is $1_{\circ}=0$. Such a structure is not a ring, for not all elements have tropical additive inverses.

Assume that the reader is familiar with the basic notations and results of the tropical Nevanlinna theory, see e.g. \cite{KLT} and \cite{Laine1}. However, for the convenience of the reader, we recall here the following basic notations.

\begin{definition}\label{zeropole}\cite{halburd}
Let $f(x)$ be a tropical meromorphic function and
\begin{equation}
\omega_{f}(x_{0})=\lim_{\varepsilon\rightarrow0^{+}}[f'(x_{0}+\varepsilon)-f'(x_{0}-\varepsilon)].
\end{equation}
If $\omega_{f}(x_{0})>0$, then $x_{0}$ is called the root $($zero$)$ of $f(x)$ with multiplicity $\omega_{f}(x_{0})$. If $\omega_{f}(x_{0})<0$, then $x_{0}$ is called the pole of $f(x)$ with multiplicity $-\omega_{f}(x_{0})$.
\end{definition}

The tropical proximity function for tropical meromorphic functions is defined as
\begin{equation}\label{p}
m(r,f):=\frac{1}{2}\left(f^{+}(r)+f^{+}(-r)\right),
\end{equation}
where $f^{+}(r):=\max\{f(x),0\}$ for $x\in \mathbf{R}$. The tropical counting function for poles in $(-r,r)$ is defined as
\begin{equation}\label{N}
N(r,f):=\frac{1}{2}\int_{0}^{r}n(t,f)dt=\frac{1}{2}\sum_{|b_{\nu}|<r}\tau_{f}(b_{\nu})(r-|b_{\nu}|),
\end{equation}
where $n(t,f)$ is the number of distinct poles of $f$ in the interval $(-r,r)$, each pole multiplied by its multiplicity $\tau_{f}$.
The tropical characteristic function $T(r,f):=m(r,f)+N(r,f)$. Following the usual classical notation, define the order of $f(x)$ as $$\rho(f):=\limsup_{r\rightarrow\infty}\frac{\log T(r,f)}{\log r},$$ and the hyper-order of $f(x)$ by $$\rho_{2}(f):=\limsup_{r\rightarrow\infty}\frac{\log\log T(r,f)}{\log r}.$$

In this paper, we typically need to work with tropical $1$-periodic meromorphic functions $\Pi (x)$ and with tropical exponential functions. Concerning $1$-periodic functions, recall their representation in terms of the sawtooth functions
$$\pi^{(a,b)}(x):=\frac{1}{a+b}\min\{ a(x-[x]),-b((x-[x])-1)\}, \quad a,b>0,$$
see \cite{KLT}, Theorem 6.7.

In what follows, we frequently shorten the notations by using the notations $\Pi (x), \Pi_{j}(x), \widetilde{\Pi}(x), \widehat{\Pi}(x)$  etc. for tropical meromorphic functions of period $1$, possibly meaning different functions at different occasions. In particular, such notations may be used to point out several independent $1$-periodic functions are to be considered simultaneously. We can also write $\Pi (x)+\Pi (x)=\Pi (x)$, $c\Pi (x)=\Pi (x)$, etc. We also use the notation $\Xi (x)$ to denote tropical meromorphic functions that are $2$-periodic and anti-$1$-periodic, again possibly meaning different functions at different occasions.

Concerning tropical exponential functions, recall first their definitions, see \cite{KLT}, Section 1.2.4:
$$e_{\alpha}(x):=\alpha^{[x]}(x-[x])+\sum_{j=-\infty}^{[x]-1}\alpha^{j}=\alpha^{[x]}\left(x-[x]+\frac{1}{\alpha-1}\right),$$
where $|\alpha|>1$ is a real number. In this case, $e_{\alpha}(x)$ is strictly increasing, and $e_{\alpha}(x)$ is a tropical entire function, since it has no poles. If then $|\beta|<1$, $\beta\neq 0$, we define
$$e_{\beta}(x):=\sum_{j=[x]}^{\infty}\beta^{j}-\beta^{[x]}(x-[x])=\beta^{[x]}\left(\frac{1}{1-\beta}-x+[x]\right).$$
If $0<\beta<1$, $e_{\beta}(x)$  is a tropical entire function as well. However, if $\beta <-1$, then $e_{\beta}(x)$ is tropical meromorphic, but not tropical entire. For more details concerning tropical exponential functions, see \cite{KLT}, Section 1.2.4. In  particular, note that tropical exponentials $y(x):=e_{\alpha}(x)$ satisfy the equation $y(x+1)=y(x)^{\otimes\alpha}(=\alpha y(x))$, for all $\alpha\neq 0,1$.

\section{Fermat type equations in the tropical setting}
The classical {\noindent\bf Fermat last theorem} that equation $x^n+y^n=1$ has no non-trivial rational solutions, when $n\geq3$, had been proved, after three centuries, by Wiles in \cite{wiles1}, see also \cite{wiles2}.

Considering $x,y$  in $x^{n}+y^{n}=1$ as elements in function fields, we land at looking equations that may be called as Fermat type functional equations. As to meromorphic solutions to the most simple case
\begin{equation}\label{n2op}
f(z)^{n}+g(z)^{n}=1,
\end{equation}
it is known that $(\ref{n2op})$ has no transcendental meromorphic solutions when $n\geq4$, while for $n=2,3$ such meromorphic solutions are easy to find.

As to meromorphic solutions of the more general case
\begin{equation}\label{n3op}
f(z)^{n}+g(z)^{n}+h(z)^{n}=1,
\end{equation}
it is known that meromorphic solutions $f,g,h$ may be found for $n\leq 6$, see \cite{ggg1}, \cite{ggg2}, and the references therein, while for $n\geq 9$ no such meromorphic solutions exist. This non-existence result has been proved by Hayman in \cite{hayman2}, where a detailed presentation for the more general situation
\begin{equation}\label{nop}
\sum_{j=1}^{n}f_{j}(z)^{k}=1
\end{equation}
is to be found. Observe that the cases $n=7,8$ remain open for (\ref{n3op}).

We now proceed to considering certain Fermat type functional equations in the tropical setting. Concerning the most simple case (\ref{n2op}), one clearly gets two corresponding tropical equations, namely
\begin{equation}\label{k2trsum}
f(x)^{\otimes k}\oplus g(x)^{\otimes k}=1
\end{equation}
and
\begin{equation}\label{k2trpr}
f(x)^{\otimes k}\otimes g(x)^{\otimes k}=1,
\end{equation}
asking in both case possible tropical meromorphic solutions $f$ and $g$, provided that $k$ is a natural number. However, as we are treating tropical meromorphic functions with real slopes, it would be natural to consider
\begin{equation}\label{r2trsum}
f(x)^{\otimes \alpha}\oplus g(x)^{\otimes \beta}=1
\end{equation}
and
\begin{equation}\label{r2trpr}
f(x)^{\otimes \alpha}\otimes g(x)^{\otimes \beta}=1,
\end{equation}
where $\alpha ,\beta$ are real numbers.

Before proceeding, recall that tropical polynomials are tropical meromorphic functions that admit finitely many roots and no poles. Equivalently, tropical polynomials may be represented in the form
\begin{equation}\label{eq3}
f(x)=\bigoplus_{i=0}^{n}(a_{i}\otimes x^{\otimes s_{i}})=a_{n}\otimes x^{\otimes s_{n}}\oplus a_{n-1}\otimes x^{\otimes s_{n-1}}\oplus\cdots\oplus a_{1}\otimes x^{\otimes s_{1}}\oplus a_{0}\otimes x^{\otimes s_{0}},
\end{equation}
where the coefficients $a_{i}$ are constants and $s_{i}$ are real numbers, $i=0,1,2,\cdots,n$ and $s_{0}<s_{1}<\cdots<s_{n}$. Evaluating tropical polynomials in the classical notation results in
\begin{equation}\label{eq4}
f(x)=\max\{a_{n}+s_{n}x,a_{n-1}+s_{n-1}x,\cdots,a_{1}+s_{1}x,a_{0}+s_{0}x\}.
\end{equation}

More generally, tropical meromorphic functions that admit no poles, may be called as tropical entire functions, and they have a series expansion
\begin{equation}\label{trent}
f(x)=\bigoplus_{n=0}^{\infty}(a_{n}\otimes x^{\otimes s_{n}})=a_{0}\otimes x^{\otimes s_{0}}\oplus a_{1}\otimes x^{\otimes s_{1}}\oplus\cdots\oplus  a_{n-1}\otimes x^{\otimes s_{n-1}} \oplus a_{n}\otimes x^{\otimes s_{n}} \cdots ,
\end{equation}
that is,
\begin{eqnarray}\label{trentcl}
f(x)=\max_{n\in \mathbf{Z}_{+}\cup\{0\}}\{a_{n}+s_{n} x\},
\end{eqnarray}
where the exponents $s_{n}$ are real numbers and $s_{0}<s_{1}<\cdots<s_{n}<\cdots$, see \cite[Chapter 2]{KLT}.

\begin{remark} To start with, observe that equation
$$f(x)^{\otimes k}\oplus g(x)^{\otimes k}=1,$$
where $k$ is a natural number, admits no non-constant tropical entire solutions. Indeed, suppose that one of $f,g$, say $f$, is non-constant. If $f$ now has no slope discontinuities, then $kf(x)>1$ for some points $x$ with $|x|$ large enough, a contradiction. Let then $x_{0}$ be a slope discontinuity of $f$, and consider the slopes $s_{j_{0}-1}$ and $s_{j_{0}}$ on both sides of $x_{0}$. Since $f$ has no poles, we have $s_{j_{0}}-s_{j_{0}-1}>0$. If $s_{j_{0}}>0$, then  $kf(x)\geq kf(x_{0})+s_{j_{0}}(x-x_{0})$ for all $x>x_{0}$, hence $kf(x)>1$ for all $x$ large enough, a contradiction. If then $s_{j_{0}}\leq 0$, we have $s_{j_{0}-1}<0$, and we have $kf(x)\geq kf(x_{0})+s_{j_{0}-1}(x-x_{0})$ for all $x<x_{0}$, and so $kf(x)>1$ for all $x$ with $|x|$ large enough, again a contradiction. On the other hand, tropical meromorphic, non-entire solutions are immediate to find. As a trivial example, take $k=1$. Then $f(x):=\min (1,-x+2)$ and $g(x):=\min (1,x)$ are solutions to $f(x)\oplus g(x)=1$.
\end{remark}

The reasoning used in the remark above applies to prove the following more general case.
\begin{theorem}\label{th3000}
There are no non-constant tropical entire functions $f_{1},\ldots ,f_{n}$ that satisfy
\begin{equation}\label{ntrop}
\bigoplus_{j=1}^{n}f_{j}(z)^{\otimes\alpha_{j}}=1,
\end{equation}
where the exponents $\alpha_{1},\ldots ,\alpha_{n}$ all are positive.
\end{theorem}

\begin{proof} Applying the reasoning in used in the preceding remark, it is immediate to see that a contradiction readily follows whenever one of the functions $f_{1},\ldots ,f_{n}$  is tropical entire and non-constant.
\end{proof}

\begin{remark} Of course, the same result follows whenever the exponents $\alpha_{1},\ldots ,\alpha_{n}$ all are negative. However, if there are different signs among the exponents $\alpha_{1},\ldots ,\alpha_{n}$, the claim obviously fails. Take, e.g., $f(x)=\frac{1}{\alpha}$ and $g(x)=\frac{1}{\beta}\oplus(x\otimes\frac{1}{\beta})$ with $\alpha>0, \beta<0$. Then we have $f(x)^{\otimes \alpha}\oplus g(x)^{\otimes \beta}=1$. Observe that $-g(x)$ is not a tropical entire function.
\end{remark}

\begin{remark} Again, it is immediate to observe that equation (\ref{ntrop}) always admits non-trivial tropical meromorphic solutions. Suppose again, for simplicity, that the exponents $\alpha_{1},\ldots ,\alpha_{n}$ all are positive. Taking now $f_{1}:=\frac{1}{\alpha_{1}}\min (1,-x+2)$, $f_{2}:=\frac{1}{\alpha_{2}}\min (1,x)$ and $\max (f_{3},\ldots ,f_{n})\leq \min (f_{1},f_{2})$, a solution $f_{1},\ldots ,f_{n}$ is at hand.
\end{remark}

We next proceed to considering tropical meromorphic solutions to equation
\begin{equation}\label{prtrop}
f(x)^{\otimes \alpha}\otimes g(x)^{\otimes \beta}=1
\end{equation}
with real exponents $\alpha ,\beta$. It is trivial to find $f$, $g$ satisfying $(\ref{prtrop})$ for any $\alpha,\beta$. Indeed, we may take, for example,  $f(x)=x\otimes\frac{1}{\alpha}$ and $g(x)=x^{\otimes-\frac{\alpha}{\beta}}$. However, finding expressions for general solutions seems to be more complicated, and we are restricting ourselves to considering equations that may be called tropical difference Fermat type functional equations. Such an equation corresponding to (\ref{prtrop}) is
\begin{equation}\label{kpp}
f(x)^{\otimes \alpha}\otimes f(x+1)^{\otimes \beta}=1,
\end{equation}
where $\alpha,\beta$ are real numbers, hence
\begin{equation}\label{kppcl}
\alpha f(x)+\beta f(x+1)=1.
\end{equation}

The key parts of this paper are then treating similar equations
\begin{equation}\label{kppclgen}
\sum_{j=0}^{s}n_{j}f(x+j)=1
\end{equation}
with $s=2$ and $s=3$.

The remaining of the paper is now being organized as follows. In Section \ref{lemmas} we collect a number of propositions that are either needed in the subsequent considerations, or might appear useful in future considerations. In Section \ref{F2}, equation (\ref{kppclgen}), i.e. the case $s=1$, is shortly treated, while the next Section \ref{F3} is devoted to considering tropical difference Fermat type equations (\ref{kppclgen}) with $s=2$. Section \ref{F4} is then treating, partially, the case $s=3$. The last two sections are describing tropical counterparts to the Hayman conjecture from complex analysis (Section \ref{Hconj}) and the Br\"{u}ck conjecture (Section \ref{Bruck}).

\section{Preliminary propositions}\label{lemmas}

We start this section by recalling the following theorem, see \cite{KLT}, Theorem 7.3 and Theorem 7.4 (and making use of the identity $e_{c}(x+1)=ce_{c}(x)$). Observe that this version is formulated for our subsequent needs.

\begin{theorem}\label{ThmA} The equation
\begin{equation}\label{eq1}
y(x+1)=y(x)^{\otimes c}
\end{equation}
with $c\in\mathbf{R}\backslash\{0\}$ admits non-constant tropical meromorphic solutions on $\mathbf{R}$ of hyper-order $\rho_{2}(f)<1$ if and only if $c=\pm1$. Moreover, if $f(x)$ is a non-constant tropical meromorphic solution to $(\ref{eq1})$, then the following representations follow:

(1) If $c=1$, then $f(x)$ is 1-periodic. Hence,
$$f(x)=\Pi (x).$$

(2) If $c=-1$, then $f(x)$ is 2-periodic, anti-1-periodic. Hence
$$f(x)=\Xi (x).$$

(3) If $c\neq\pm1$, then all solutions of $(\ref{eq1})$ are finite linear combinations of tropical exponentials of type $f(x)=e_{c}(x-b)$, where $b\in [0,1)$.
\end{theorem}

\begin{remark}\label{rem305} In what follows, we use the notation $L_{b}(e_{c}(x-b))$ for finite linear combinations $\sum_{j=1}^{q}\beta_{j}e_{c}(x-b_{j})$, where $c$ is fixed and $b=\{ b_{1},\ldots ,b_{q}\}\subset [0,1)$. In particular, if $A$ is a constant, then $L_{b}(Ae_{c}(x-b))=L_{b}(e_{c}(x-b))$, since, $e_{c}(x+1-b)=ce_{c}(x-b)$, $L_{b}(e_{c}(x+1-b))=L_{b}(e_{c}(x-b))$.
\end{remark}

\begin{proposition}\label{P1}
All tropical meromorphic solutions to
$$f(x+1)-f(x)=c, \quad c\in\mathbf{R},$$
are of the form $f(x)=\Pi (x)+cx.$
\end{proposition}

\begin{remark}\label{rem31}
As for the trivial proof of this proposition, see \cite{KLT}, p. 157.
\end{remark}

In what follows, we use the notation $\Pi_{0}$ for $1$-periodic tropical meromorphic functions that satisfy $\Pi_{0}(0)=0$.

\begin{proposition}\label{P11}
All tropical meromorphic solutions to
\begin{equation}\label{E11}
f(x+1)-f(x)=\Pi_{0}(x),
\end{equation}
take the form
\begin{equation}\label{S11}
f(x)=\Pi (x)+\Phi (x,\Pi_{0})
\end{equation}
where
\begin{equation}\label{F11}
\Phi (x,\Pi_{0}):=[x]\Pi_{0}(x).
\end{equation}
\end{proposition}

\begin{proof} It is straightforward to see that $\Phi (x+1,\Pi_{0})-\Phi (x,\Pi_{0})=\Pi_{0}(x)$. Therefore, $\Phi (x,\Pi_{0})$ is a special solution to (\ref{E11}). It remains to verify that $\Phi (x,\Pi_{0})$ is continuous and piecewise linear; this is immediate. On the other hand, if $F(x)$ is an arbitrary solution to (\ref{E11}), then it is trivial to see that $F(x)-\Phi (x,\Pi_{0})$ is $1$-periodic.
\end{proof}

\begin{corollary}\label{P11cor}
All tropical meromorphic solutions to
\begin{equation}\label{E11cor}
f(x+1)-f(x)=\Pi (x),
\end{equation}
where $\Pi$ is $1$-periodic tropical meromorphic such that $\Pi (0)=d$, take the form
\begin{equation}\label{S11cor}
f(x)=\widetilde{\Pi}(x)+\Phi (x,\Pi)+dx
\end{equation}
where
\begin{equation}\label{F11cor}
\Phi (x,\Pi):=[x](\Pi (x) -d).
\end{equation}
\end{corollary}

\begin{remark}\label{P11correm} Observe that Remark 7.2 in \cite{KLT} fails by Proposition \ref{P11}. Note that $x\Pi_{0}(x)$ satisfies equation (\ref{E11}), and therefore, $g(x):=x\Pi_{0}(x)-(\Phi (x,\Pi_{0}))$ is $1$-periodic. However, $x\Pi_{0}(x)$ is not tropical meromorphic by its non-linearity, hence $g(x)$ is not tropical meromorphic as well. Moreover, we remark here that Theorem 7.7 in \cite{KLT} becomes incomplete, see Theorem \ref{ThmB}(1) below.
\end{remark}

As an illustration, see the graph of $\Phi (x,\Pi_{0})$ in Fig \ref{goodness}, when $\Pi_{0}(x)=\pi^{(1,1)}(x)$.

\begin{figure}[h]
\begin{minipage}[t]{1\linewidth}
\centering
\includegraphics[width=0.5\textwidth]{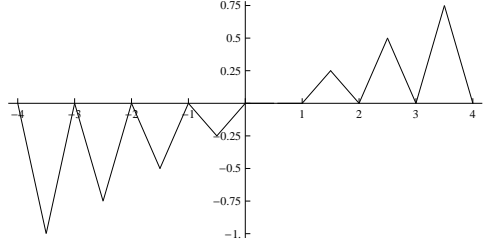}
\caption{$\Phi(x,\Pi_{0}),x\in[-4,4], \Pi_{0}(x)=\pi^{(1,1)}(x)$}\label{goodness}
\end{minipage}
\hspace{0.2cm}
\end{figure}

As for the subsequent reasoning, we need to proceed by proving

\begin{proposition}\label{P111}
All tropical meromorphic solutions to
\begin{equation}\label{E110}
f(x+1)-f(x)=\Phi (x,\Pi_{0}),
\end{equation}
with the $1$-periodic function $\Phi (x,\Pi_{0}):=[x]\Pi_{0}(x)$ defined as in Proposition \ref{P11}, take the form
\begin{equation}\label{S11}
f(x)=\Pi (x)+\Theta (x,\Pi_{0}),
\end{equation}
where
\begin{equation}\label{F11}
\Theta (x,\Pi_{0}):=(1+\frac{[x]([x]-1)}{2})\Pi_{0}(x).
\end{equation}
\end{proposition}

\begin{proof} This is an elementary computation.
\end{proof}

To illustrate the situation, look at the graph of $\Theta (x,\Pi_{0})$ in Fig \ref{sdd}, with $\Pi_{0}(x)=\pi^{(1,1)}(x)$ again.

\begin{figure}[h]
\begin{minipage}[t]{1\linewidth}
\centering
\includegraphics[width=0.5\textwidth]{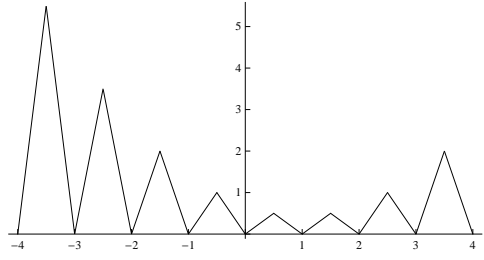}
\caption{$\Theta(x,\Pi_{0}),x\in[-4,4], \Pi_{0}(x)=\pi^{(1,1)}(x)$}\label{sdd}
\end{minipage}
\hspace{0.2cm}
\end{figure}

\begin{corollary}\label{P111cor}
All tropical meromorphic solutions to
\begin{equation}\label{E11cor}
f(x+1)-f(x)=\Phi (x,\Pi)
\end{equation}
with the $1$-periodic function $\Phi (x,\Pi):=[x](\Pi (x)-\Pi (0))$ defined as in Corollary \ref{P11cor}, take the form
\begin{equation}\label{S11cor}
f(x)=\Pi (x)+\Theta (x,\Pi),
\end{equation}
where
\begin{equation}\label{F11}
\Theta (x,\Pi):=(1+\frac{[x]([x]-1)}{2})(\Pi (x)-\Pi (0)).
\end{equation}
\end{corollary}

\begin{proposition}\label{P1111}
Tropical meromorphic solutions of
\begin{equation}\label{E1111}
f(x+1)-f(x)=\Theta(x,\Pi_{0})
\end{equation}
satisfy
\begin{equation}\label{S1111}
f(x)=\Pi(x)+\Omega(x,\Pi_{0}),
\end{equation}
where
$\Omega(x)=\left([x-1]+\frac{[x][x-1][x-2]}{6}\right)\Pi_{0}(x)$.
\end{proposition}

As an example for the graph $\Omega(x,\Pi_{0})$, see Fig $\ref{sdddd}$, where again $\Pi_{0}(x)=\pi^{(1,1)}(x)$.

\begin{figure}[h]
\begin{minipage}[t]{1\linewidth}
\centering
\includegraphics[width=0.5\textwidth]{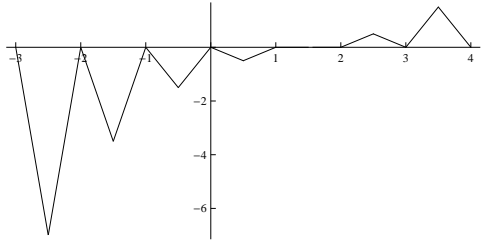}
\caption{$\Omega(x,\Pi_{0}),x\in[-3,4],\Pi_{0}(x)=\pi^{(1,1)}(x)$}\label{sdddd}
\end{minipage}
\hspace{0.2cm}
\end{figure}

Similarly as before, we again obtain

\begin{corollary}\label{P1111cor}
Tropical meromorphic solutions of
\begin{equation}\label{E1111cor}
f(x+1)-f(x)=\Theta(x,\Pi_{0})
\end{equation}
satisfy
\begin{equation}\label{S1111cor}
f(x)=\Pi(x)+\Omega(x,\Pi_{0}),
\end{equation}
where
$\Omega(x)=\left([x-1]+\frac{[x][x-1][x-2]}{6}\right)(\Pi_{0}(x)-\Pi_{0}(0))$.
\end{corollary}

We next introduce the notation $\Xi_{0}(x)$ to mean anti-$1$-periodic, $2$-periodic tropical meromorphic functions such that $\Xi_{0}(0)=0$. We now proceed to prove

\begin{proposition}\label{P1plus}
Tropical meromorphic solutions to
\begin{equation}\label{E1plus}
f(x+1)+f(x)=\Xi_{0}(x)
\end{equation}
satisfy
\begin{equation}\label{S1plus}
f(x)=\Xi (x)-[x]\Xi_{0}(x).
\end{equation}
\end{proposition}

\begin{proof} Elementary computation again verifies that $-[x]\Xi_{0}(x)$ is a special solution to (\ref{E1plus}). Moreover, if $F(x)$ is an arbitrary solution to (\ref{E1plus}), it is an easy exercise to see that $F(x)+[x]\Xi_{0}(x)$ is an anti-$1$-periodic, $2$-periodic tropical meromorphic function.
\end{proof}

As for the general case $\widetilde{\Xi}(x)$ of anti-$1$-periodic, $2$-periodic tropical meromorphic functions, it immediately follows by continuity and anti-$1$-periodicity that there exists $x_{0}$ such that $\widetilde{\Xi}(x_{0})=0$. Therefore, we obtain

\begin{proposition}\label{P2plus}
Tropical meromorphic solutions to
\begin{equation}\label{E2plus}
f(x+1)+f(x)=\widetilde{\Xi}(x),
\end{equation}
such that $\widetilde{\Xi}(x_{0})=0$, take the form
\begin{equation}\label{S2plus}
f(x)=\Xi (x)-[x-x_{0}]\widetilde{\Xi}(x).
\end{equation}
\end{proposition}

\begin{proof} Defining $\widetilde{\Xi_{0}}(x):=\widetilde{\Xi}(x+x_{0})$ and $g(x):=f(x+x_{0})$, equation (\ref{E2plus}) takes the form
$$g(x+1)+g(x)=\widetilde{\Xi_{0}}(x),$$
where $\widetilde{\Xi_{0}}(0)=0$. By Proposition \ref{P1plus},
$$g(x)=\Xi (x)-[x]\widetilde{\Xi_{0}}(x)=\Xi (x)-[x]\widetilde{\Xi}(x+x_{0}),$$
hence
$$f(x)=g(x-x_{0})=\Xi (x-x_{0})-[x-x_{0}]\widetilde{\Xi}(x).$$
\end{proof}

\begin{proposition}\label{P3}
Given a $2$-periodic, anti-$1$-periodic tropical meromorphic function $\Xi (x)$, all tropical meromorphic solutions to $$f(x+1)-f(x)=\Xi (x)$$
are of the form $f(x)=\Pi (x)-\frac{1}{2}\Xi (x)$.
\end{proposition}

\begin{proof} It is a trivial computation to verify that $f(x)=-\frac{1}{2}\Xi (x)$ is a special solution to $f(x+1)-f(x)=\Xi (x)$. One can also immediately see that whenever $f(x)$ is an arbitrary tropical meromorphic solution, then $f(x)+\frac{1}{2}\Xi (x)$ is $1$-periodic.
\end{proof}

\begin{proposition}\label{P31}
Given a $2$-periodic, anti-$1$-periodic tropical meromorphic function $\Xi (x)$ such that $\Xi (x_{0})=0$, then all tropical meromorphic solutions to $$f(x+1)-f(x)=[x-x_{0}]\Xi (x)$$ are of the form $f(x)=\Pi (x)-\frac{1}{2}[x-x_{0}]\Xi (x)+\frac{1}{4}\Xi (x).$
\end{proposition}

\begin{proof} This is a straightforward computation.
\end{proof}

\begin{proposition}\label{P12}
All tropical meromorphic solutions to
$$f(x+1)+f(x)=\Pi (x),$$
where $\Pi (x)$ is tropical meromorphic and $1$-periodic, are of the form $f(x)=\Xi (x)+\frac{\Pi(x)}{2}.$
\end{proposition}

\begin{proof} Clearly, the solutions to $f(x+1)+f(x)=0$ are $2$-periodic, anti-$1$-periodic functions, while $\frac{\Pi (x)}{2}$ is a special solution to $f(x+1)+f(x)=\Pi (x)$.
\end{proof}

\begin{proposition}\label{P2}
Provided $a\neq\alpha$, all tropical meromorphic solutions to
$$f(x+1)-af(x)=e_{\alpha}(x)$$
are of the form $f(x)=G(x)+\frac{1}{\alpha -a}e_{\alpha}(x)$, where $G(x)$ stands for the solutions of the homogeneous equation $f(x+1)-af(x)=0$, as given in Theorem \ref{ThmA}.
\end{proposition}

\begin{proof} To determine the special solution of the form $Ke_{\alpha}(x)$, it is sufficient to substitute this into $f(x+1)-af(x)=e_{\alpha}(x)$, and recall that $e_{\alpha}(x+1)=\alpha e_{\alpha}(x)$. Moreover, if $F(x)$ is an arbitrary tropical meromorphic solution to $f(x+1)-af(x)=e_{\alpha}(x)$, then an elementary computation shows that $G(x):=F(x)-\frac{1}{\alpha -a}e_{\alpha}(x)$ satisfies $G(x+1)-aG(x)=0$.
\end{proof}

\begin{remark}\label{P2rem} Clearly, all tropical meromorphic solutions to
$$f(x+1)-af(x)=Ae_{\alpha}(x)$$
with a constant multiplier $A$ are of the form $f(x)=G(x)+\frac{A}{\alpha -a}e_{\alpha}(x)$.
\end{remark}

\begin{proposition}\label{P25}
Provided that $|\alpha |\neq 1$, and that $\alpha <0$, then all tropical meromorphic solutions to
$$f(x+1)-\alpha f(x)=e_{\alpha}(x)$$
are of the form $f(x)=L_{b}(e_{\alpha}(x-b))+\frac{1}{\alpha}[x-x_{0}]e_{\alpha}(x)$, where $e_{\alpha}(x_{0})=0$.
\end{proposition}

\begin{proof} It is an easy observation to see that $[x-x_{0}]e_{\alpha}(x)$ is tropical meromorphic. What remains is a trivial computation.
\end{proof}

\begin{remark}\label{P25rem} Observe that we may assume that $x_{0}=\frac{1}{1-\alpha}\in (0,1)$ in this proposition. Indeed, all roots of the tropical exponential $e_{\alpha}(x)$ form the set $x_{0}+\mathbf{Z}$. If now $x_{1},x_{2}\in x_{0}+\mathbf{Z}$, then their difference is a solution to $f(x+1)-\alpha f(x)=0$ as one may immediately see.
\end{remark}

\begin{remark}\label{P25rem2} It remains open whether equation $f(x+1)-\alpha f(x)=e_{\alpha}(x)$ admits tropical meromorphic solutions, if $|\alpha |\neq 1$, and $\alpha >0$. Indeed, if $f$ is such a solution, then $f(x+2)-2\alpha f(x+1)+\alpha^{2}f(x)=0$. By Theorem \ref{ThmB}(4), $f(x)\in L_{b}(e_{\alpha}(x-b))$ are solutions to $f(x+2)-2\alpha f(x+1)+\alpha^{2}f(x)=0$, but not to $f(x+1)-\alpha f(x)=e_{\alpha}(x)$. Existence of tropical meromorphic solutions of different type to $f(x+2)-2\alpha f(x+1)+\alpha^{2}f(x)=0$ remains open, see Remark \ref{ThmBrem}.
\end{remark}

We next define a special tropical meromorphic function $\Psi (x)$ as follows:
\begin{equation}
\Psi(x):=\left\{
 \begin{array}{r@{\;,\;}l}
\sum_{j=0}^{[x]}\max(0,x-j) &x\geq 1,\\
\sum_{j=0}^{[x]+1}\max(0,x-j) & 0\leq x<1,\\
x+\sum_{j=[x]}^{0}\max(0,-x+j) & x<0.\\
 \end{array}
 \right.
\end{equation}
The idea to apply such a tropical meromorphic function goes back to Tohge, see \cite{Tohge}, p. 133. The importance of this function becomes immediately clear below. For practical computations, $\Psi (x)$ may be represented as follows:
\begin{equation}\label{psi2}
\Psi (x)=([x]+1)x-\frac{1}{2}[x]([x]+1).
\end{equation}
It is an elementary computation to see that these two representations for $\Psi (x)$ are identical.  As for the graph of $\Psi (x)$ for $x\in[-3,3]$, see Fig \ref{goodnesskill}.

\begin{figure}[h]
\begin{minipage}[t]{1\linewidth}
\centering
\includegraphics[width=0.5\textwidth]{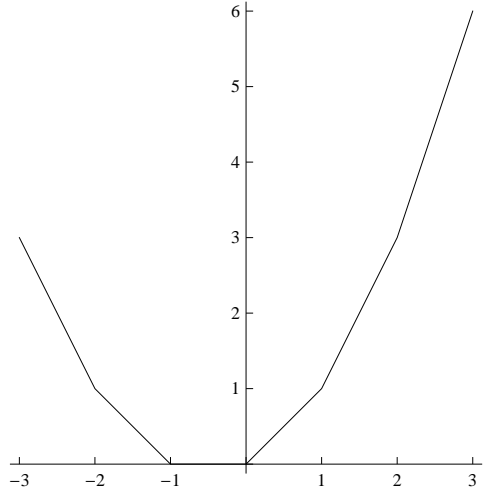}
\caption{$\Psi(x),x\in[-3,3]$}\label{goodnesskill}
\end{minipage}
\hspace{0.2cm}
\end{figure}

\begin{proposition}\label{P4} The tropical meromorphic function $\Psi (x)$ is tropical entire of order $\rho(\Psi)=2$ and satisfies the difference equation $\Psi(x)-\Psi(x-1)=x$.
\end{proposition}

\begin{proof} It is immediate to verify that $\Psi (x)$ satisfies the asserted difference equation. Indeed,
$$\Psi (x)-\Psi (x-1)=([x]+1)x-\frac{1}{2}[x]([x]+1)-([x](x-1)-\frac{1}{2}[x]([x]-1))=x.$$
Next we see that $\Psi(x)$ is a continuous piecewise linear function: Suppose that $0<\varepsilon <1$. Then
$$\Psi ([x]+\varepsilon )=([x]+1)([x]+\varepsilon )-\frac{1}{2}[x]([x]+1)\rightarrow ([x]+1)[x]-\frac{1}{2}[x]([x]+1)$$
as $\varepsilon\rightarrow 0$ and
$$\Psi ([x]-\varepsilon )=[x]([x]-\varepsilon )-\frac{1}{2}[x]([x]-1)\rightarrow [x][x]-\frac{1}{2}[x]([x]-1);$$
the limits are equal as one may readily check. The difference of the slopes around each integer, say $[x]$, equals to $[x]+1-[x]=1$, hence $\Psi (x)$ has no poles, being a tropical entire function. By elementary computation from $T(r,\Psi )=m(r,\Psi )=\frac{1}{2}(([r]+1)r-\frac{1}{2}[r]([r]+1)+([-r]+1)(-r)-\frac{1}{2}[-r]([-r]+1))$, it readily follows that $\rho(\Psi)=2$, completing the proof.
\end{proof}

\begin{proposition}\label{P5}
All tropical meromorphic solutions to
$$f(x+1)-f(x)=cx, \quad c\in\mathbf{R}$$
are of the form $f(x)=\Pi (x)+c(\Psi (x)-x).$
\end{proposition}

\begin{proof} It is, again, straightforward to verify that $f(x)=c(\Psi (x)-x)$ is a special solution by substituting and making use of the identity $\Psi (x+1)-\Psi (x)=x+1$.
\end{proof}

\begin{proposition}\label{P6}
All tropical meromorphic solutions to
\begin{equation}\label{1}
f(x+1)-f(x)=\Psi(x)
\end{equation}
are of the form $f(x)=\Pi (x)+\Upsilon (x),$
where $\Upsilon (x):=\frac{1}{6}[x]([x]+1)(2(x-[x])+x-1).$
\end{proposition}

\begin{proof} Recalling the definition of $\Psi$ in (\ref{psi2}), this is a straightforward computation.
\end{proof}

\section{Tropical difference Fermat type equations with two terms}\label{F2}

In this short section, we solve equation (\ref{kppcl}):
$$\beta f(x+1)+\alpha f(x)=1$$
for tropical meromorphic solutions:
\begin{theorem}\label{thlloo}
The tropical meromorphic solutions of $(\ref{kppcl})$ satisfy:\\
(i) If $\alpha=\beta$, then $f(x)=\Xi (x)+\frac{1}{2\beta},$\\
(ii) if $\alpha=-\beta$, then $f(x)=\Pi (x)+\frac{1}{\beta}x,$\\
(iii) if $\alpha\neq\pm\beta$, then $f(x)=L_{b}(e_{-\frac{\alpha}{\beta}}(x-b))+\frac{1}{\alpha +\beta}$.
\end{theorem}

\begin{proof} If $\alpha=\beta$, then (\ref{kppcl}) takes the form
$$f(x+1)+f(x)=\frac{1}{\beta}.$$
The claim now follows from Proposition \ref{P12}.

If $\alpha=-\beta$, then $f(x+1)-f(x)=\frac{1}{\beta}$, and an immediate application of Proposition \ref{P1} proves the claim.

As for the general case $\alpha\neq\pm\beta$, we may rewrite $(\ref{kppcl})$ as
$$f(x+1)+\frac{\alpha}{\beta}f(x)=1.$$
Now, the general solution $L_{b}(e_{-\frac{\alpha}{\beta}}(x-b))$ to the homogeneous equation $f(x+1)+\frac{\alpha}{\beta}f(x)=0$ follows from Theorem \ref{ThmA}. Finally, a special constant solution $\frac{1}{\alpha +\beta}$ to $f(x+1)+\frac{\alpha}{\beta}f(x)=1$ is nothing but a trivial observation.
\end{proof}

\section{Tropical difference Fermat type equations with three terms}\label{F3}
We next proceed to considering tropical difference Fermat type functional equations with three terms, such as
\begin{equation}\label{bu}
y(x)^{\otimes n}\otimes y(x+1)^{\otimes m}\otimes y(x+2)^{\otimes p}=1.
\end{equation}
In the classical notation, we have
\begin{equation}\label{bucl}
ny(x)+my(x+1)+py(x+2)=1.
\end{equation}
First observe that if $n=0$ or $p=0$, (\ref{bucl}) reduces back to the preceding section. However, if $m=0$, we have $ny(x)+py(x+2)=1$, and this needs to be considered separately, whenever $n+p\neq 0$, see below.

\subsection{The case $n+m+p=0$.} If $n+m+p=0$, then equation (\ref{bucl}) takes the form
\begin{equation}\label{bu1}
F(x+1)-\frac{n}{p}F(x)=\frac{1}{p}
\end{equation}
by setting $F(x)=y(x+1)-y(x)$.

If now $n=p$, then, by Proposition \ref{P1},
$$F(x)=\Pi (x)+\frac{1}{p}x.$$
Recalling then Proposition \ref{P11} and Proposition \ref{P5}, we conclude that all tropical meromorphic solutions to (\ref{bucl}) may be written in the form
\begin{equation}\label{buclS}
y(x)=\widetilde{\Pi} (x)+\Phi (x,\Pi)+\Phi (0,\Pi)x+\frac{1}{p}(\Psi (x)-x).
\end{equation}

If next $n=-p$, then we have $m=0$. Reversed, if $m=0$, then $n=-p$. In this case,
$$F(x+1)+F(x)=\frac{1}{p},$$
hence $F(x)=y(x+1)-y(x)=\Xi (x)+\frac{1}{2p}$ by Proposition \ref{P12}. Relying now on Proposition \ref{P3} and Proposition \ref{P5}, we obtain
$$y(x)=\Pi (x)-\frac{1}{2}\Xi (x)+\frac{1}{2p}x.$$

It remains to consider the case $n\neq\pm p$. By Theorem \ref{thlloo}, we obtain
$$F(x)=L_{b}(e_{n/p}(x-b))+\frac{1}{p-n}.$$
Therefore, by linearity and Proposition \ref{P2}, all tropical meromorphic solutions to (\ref{bucl}) now take the form
\begin{equation}\label{buclS2}
y(x)=\Pi (x)+\frac{p}{n-p}L_{b}(e_{n/p}(x-b))+\frac{1}{p-n}x.
\end{equation}

\subsection{The case $n+m+p\neq 0$.} Define $F(x):=y(x)-\frac{1}{n+m+p}$. Thus, $(\ref{bucl})$ takes the form
\begin{equation}\label{buhao}
nF(x)+mF(x+1)+pF(x+2)=0,
\end{equation}
and all solutions to (\ref{bucl}) follow as soon as all solutions to (\ref{buhao}) have been obtained. Writing now $c:=-\frac{m}{p}$ and $d:=\frac{n}{p}$, $(\ref{buhao})$ may be written as
\begin{equation}\label{haolei}
 F(x+1)-cF(x)+dF(x-1)=0.
\end{equation}
Ideas to solving (\ref{haolei}) can be found in \cite{Laine1} and in \cite{KLT}, where the case $d=1$ has been treated. However, by Proposition \ref{P11} and Proposition \ref{P111} above, the statements in Theorem 7.7 and Theorem 7.8 in \cite{KLT} are incomplete. It had been pointed out in \cite{Laine1} that a similar method could be used to solving (\ref{haolei}) for arbitrary real numbers $d\neq 0,1$. Due to the incomplete reasoning in \cite{KLT}, we prove the following theorem in detail.

\begin{theorem}\label{ThmB}
Tropical meromorphic solutions $F$ of equation (\ref{haolei}) exist as follows:

(1) If $c=2, d=1$, then all tropical meromorphic solutions $F$ to (\ref{haolei}) are linear combinations of $\Pi (x)$, of $L(x)\equiv x$ and all tropical meromorphic functions $\Phi (x,\Pi_{0}):=[x]\Pi_{0}(x)$, where $\Pi_{0}$ is an arbitrary tropical meromorphic function that is $1$-periodic and satisfies $\Pi_{0}(0)=0$.

(2) If $c=-2, d=1$, then all tropical meromorphic solutions $F$ to (\ref{haolei}) are linear combinations of tropical meromorphic functions that are either $2$-periodic and anti-$1$-periodic, or of type $F(x)=[x-x_{0}]\Xi (x)$, where $x_{0}\in\mathbf{R}$ is arbitrary $\Xi$ is an arbitrary $2$-periodic, anti-$1$-periodic meromorphic function such that $\Xi(x_{0})=0$.

(3) If $c^{2}-4d=0, d\neq 0,1$, and $c<0$, then all tropical meromorphic solutions to (\ref{haolei}) are either linear combinations of functions of type $F(x)=L_{b}(e_{c/2}(x-b))$, or linear combinations of functions of type (2) $F(x)=[x-b_{0}]L_{b}(e_{c/2}(x-b))$, where $e_{c/2}(b_{0}-b)=0$.

(4) If $c^{2}-4d=0, d\neq 0,1$, and $c>0$, then all tropical meromorphic functions of type $F(x)=L_{b}(e_{c/2}(x-b))$ are solutions to (\ref{haolei}).

(5) If  $c^{2}-4d>0, d\neq 0,1$, then all tropical meromorphic solutions $F$ to (\ref{haolei}) are linear combinations of functions that are either in $L_{b}(e_{\alpha}(x-b))$ or in $L_{b}(e_{\beta}(x-b))$, where  $\alpha ,\beta$ are the roots of $\lambda^{2}-c\lambda +d=0$.
\end{theorem}

\begin{remark}\label{ThmBrem} (1) In Theorem \ref{ThmB}(4), it remains open to us, whether the solutions described therein are indeed all tropical meromorphic solutions to (\ref{haolei}).

(2) Observe that the case $c^{2}-4d<0$ remains open, if $d\neq 1$, compare to \cite{KLT}, p. 165. Following the ideas in \cite{KLT}, see also \cite{Laine1}, for the case $d=1$, natural candidates for solutions to (\ref{haolei}) in the case $d\neq 1$ might be
$$F_{1}(x)=r^{[x]}\cos(\theta[x])(x-[x])+r^{[x]}\frac{\cos(\theta[x])(\cos\theta-1)+\sin(\theta[x])\sin\theta}{2(1-\cos\theta)},$$
and
$$F_{2}(x)=r^{[x]}\sin(\theta[x])(x-[x])+r^{[x]}\frac{\sin(\theta[x])(\cos\theta-1)+\cos(\theta[x])\sin\theta}{2(1-\cos\theta)},$$
where $r^{2}=d\neq 1$, $c=2r\cos\theta$, $\theta\in (0,\pi )$. However, it is straightforward to see that $r^{[x]}$ is discontinuous at integers, whenever $d\neq 1$, hence $F_{1}(x)$, $F_{2}(x)$ are not tropical meromorphic.
\end{remark}

\begin{proof} To prove Case (1), it is elementary to verify that $\Pi (x)$ and $L(x)$ are, separately, solutions to (\ref{haolei}) as well as $\Phi (x,\Pi_{0})$ for an arbitrary $\Pi_{0}$. On the other hand, is $f(x)$ is an arbitrary non-vanishing tropical meromorphic solution to (\ref{haolei}), that is not $1$-periodic, we anyway have $f(x+1)-f(x)=f(x)-f(x-1)$. Hence, $\Pi_{0}(x):=f(x)-f(x-1)$ is non-vanishing and $1$-periodic, meaning that $f(x)=\Phi (x,\Pi_{0})+\Pi_{0}(0)x$ by Proposition \ref{P11}.

As to Case (2), it is immediate to verify that all $2$-periodic, anti-$1$-periodic tropical meromorphic functions are solutions to (\ref{haolei}) as well as tropical meromorphic solutions of type $F(x)=[x-x_{0}]\Xi (x)$ such that $\Xi (x_{0})=0$. Moreover, if $F$ is an arbitrary solution, then we may write (\ref{haolei}) in the form
$$F(x+1)+F(x)=-(F(x)+F(x-1)).$$
If now $T(x):=F(x)+F(x-1)$ vanishes, then $F$ has to be $2$-periodic and anti-$1$-periodic. Otherwise, $T$ is $2$-periodic, anti-$1$-periodic, and $F$ has to be of the asserted form by Proposition \ref{P2plus}.

We next proceed to prove Case (3). Since $c^{2}-4d=0$ and $d\neq1$, then $\frac{c}{2}(\neq1)$ is the root of $\lambda^{2}-c\lambda+d=0$.  Equation (\ref{haolei}) can be stated as
\begin{equation}\label{haolei1}
F(x+1)-\frac{c}{2}F(x)=\frac{c}{2}(F(x)-\frac{c}{2}F(x-1)).
\end{equation}
Denoting $T(x):=F(x)-\frac{c}{2}F(x-1)$, we see that $T(x)$ satisfies $T(x+1)=\frac{c}{2}T(x)$. If $T$ vanishes, then
$$F(x)-\frac{c}{2}F(x-1)=0,$$
meaning that $F(x)$ must be of type $F(x)=L_{b}(e_{c/2}(x-b))$. Otherwise, $T(x)$ has to be of type $L_{b}(e_{c/2}(x-b))$ as well and we may write
\begin{equation}\label{haolei2}
F(x)-\frac{c}{2}F(x-1)=\sum_{j=1}^{q}\alpha_{j}e_{\frac{c}{2}}(x-b_{j}).
\end{equation}
Suppose first $c<-2$ and recall that
$$e_{\frac{c}{2}}(x-b_{j})=(\frac{c}{2})^{[x-b_{j}]}(x-b_{j}-[x-b_{j}]+\frac{2}{c-2}).$$
To determine a point $b_{0,j}$ where $e_{\frac{c}{2}}(b_{0,j}-b_{j})=0$, we may take $b_{0,j}$ so that $0\leq b_{0,j}-x_{j}<1$. Then $[b_{0,j}-b_{j}]=0$, and we get $b_{0,j}=b_{j}-\frac{2}{c-2}$. It is now a straightforward computation to verify that $F(x):=[x-b_{0,j}]e_{\frac{c}{2}}(x-b_{j})$ is tropical meromorphic and satisfies equation (\ref{haolei}). A similar reasoning may be used to treat the case $-2<c<0$.

Case (4) is nothing else than a straightforward computation.

As to Case (5), it is again straightforward to verify the all functions mentioned in the claim are solutions to (\ref{haolei}). To see that an arbitrary solution to (\ref{haolei}) is a linear combination of functions in the claim, it is sufficient to refer to the proof of \cite{KLT}, Theorem 7.9, the reasoning therein carries over verbatim.
\end{proof}

\section{Tropical difference Fermat type functional equations with four terms}\label{F4}

We proceed to considering tropical difference Fermat type functional equations with four terms, such as
\begin{equation}\label{eq4}
y(x)^{\otimes n}\otimes y(x+1)^{\otimes m}\otimes y(x+2)^{\otimes p}\otimes y(x+3)^{\otimes q}=1,
\end{equation}
where $n,m,p,q$ are non-zero real numbers. Using classical notations, equation (\ref{eq4}) equals to
\begin{equation}\label{eq5}
ny(x)+my(x+1)+py(x+2)+qy(x+3)=1.
\end{equation}

The considerations can be divided in two parts, assuming that either $n+m+p+q=0$, or $n+m+p+q\neq 0$.
\subsection{The case $n+m+p+q=0$.} This subsection splits in two parts, assuming that either $3n+2m+p=0$ or that $3n+2m+p\neq 0$. Before proceeding, observe that our assumption $n,m,p,q\neq 0$ actually could be deleted. Indeed, if $n=0$, or if $q=0$, then (\ref{eq5}) returns back to the considerations in the preceding section. If then $m=0$, resp. $p=0$, then denoting $g(x):=y(x+2)-y(x+3)$, (\ref{eq5}) reduces to $(n+p)g(x)+ng(x-1)+ng(x-2)=1$, resp. to $(n+m)g(x)+(n+m)g(x-1)+ng(x-2)=1$, and we may again proceed by the arguments in the preceding sections.

\begin{theorem}\label{th1} If $n+m+p+q=0$ and $3n+2m+p=0$, then the following conclusions hold for tropical meromorphic solutions $y(x)$ of equation (\ref{eq5}):

If $2n+m=n$, then
$$y(x)=\widetilde{\Pi} (x)+\Phi (x,\Pi )-\frac{1}{4}\Xi (x)+\frac{1}{2n}(\Psi (x)-x);$$

If $2n+m=-n$, then
$$y(x)=\widehat{\Pi}(x)+\Phi (x,\widetilde{\Pi})-\Theta (x,\Pi )-\frac{1}{n}\Upsilon (x)+\frac{3}{n}\Psi (x)+(\Pi (0)-\frac{3}{n})x.$$

If $2n+m\neq\pm n, 0$, then
$$y(x)=\widetilde{\Pi}(x)+\Phi (x,\Pi)+(\Pi(0)-\frac{2}{3n+m})x+\frac{1}{3n+m}\Psi (x)+L_{b}(e_{-n/(2n+m)}(x-b)).$$
\end{theorem}

\begin{proof}  To prepare the separate proofs for each of the subcases, denote $g(x):=y(x+2)-y(x+3)$. Equation (\ref{eq5}) then takes the form
\begin{equation}\label{eq51}
(n+m+p)g(x)+(n+m)g(x-1)+ng(x-2)=1
\end{equation}
and, making use of the assumption that $3n+2m+p=0$, further
\begin{equation}\label{eq52}
(n+m)(g(x-1)-g(x))+n(g(x-2)-g(x))=1.
\end{equation}

Denoting now $H(x):=g(x-1)-g(x)$, we obtain
\begin{equation}\label{eq53}
(2n+m)H(x)+nH(x-1)=1.
\end{equation}

We now proceed to considering the subcases separately (observing first that if $2n+m=0$, then $q=0$, contradicting our assumptions):

(1) In the case $2n+m=n$, we have $n=-m$ and equation (\ref{eq53}) takes the form
$$H(x)+H(x-1)=\frac{1}{n},$$
resulting in $H(x)=\Xi (x)+\frac{1}{2n}$ by Proposition \ref{P12}. Hence
$$g(x+1)-g(x)=-\frac{1}{2n}+\Xi (x).$$
Solving termwise by linearity, and making use of Propositions \ref{P1} and \ref{P11}, we obtain
$$g(x)=y(x+2)-y(x+3)=\Pi (x)-\frac{1}{2}\Xi (x)-\frac{1}{2n}x.$$
Therefore,
$$y(x+1)-y(x)=\Pi (x)+\frac{1}{2}\Xi (x)+\frac{1}{2n}x.$$
Applying now Propositions \ref{P1}, \ref{P11}, \ref{P3} and \ref{P5}, we get
$$y(x)=\widetilde{\Pi} (x)+\Phi (x,\Pi )-\frac{1}{4}\Xi (x)+\frac{1}{2n}(\Psi (x)-x).$$
Observe that (\ref{eq5}) reduces in this case to $$y(x)-y(x+1)-y(x+2)+y(x+3)=\frac{1}{n}.$$

\medskip

(2) In the case $2n+m=-n$, we have $m=-3n$ and equation (\ref{eq53}) takes the form
$$H(x)-H(x-1)=-\frac{1}{n}$$
being solved by $H(x)=g(x-1)-g(x)=\Pi (x)-\frac{1}{n}x$. Therefore,
$$g(x+1)-g(x)=\Pi (x)+\frac{1}{n}x,$$
hence
$$g(x)=y(x+2)-y(x+3)=\widetilde{\Pi}(x)+\Phi (x,\Pi )+\frac{1}{n}(\Psi (x)-x)$$
and
$$y(x+1)-y(x)=\widetilde{\Pi}(x)-\Phi (x-2,\Pi )-\frac{1}{n}\Psi (x-2)+\frac{1}{n}x.$$
Since $\Phi (x-2,\Pi )-\Phi (x,\Pi )=-2(\Pi (x-2)-\Pi (0))$ is $1$-periodic, and $\Psi (x-2)=\Psi (x)-2x+1$, we may further write
$$y(x+1)-y(x)=\widetilde{\Pi}(x)-\Phi (x,\Pi )-\frac{1}{n}\Psi (x)+\frac{3}{n}x.$$
Making now use of Corollary \ref{P11cor}, Corollary \ref{P111cor}, Proposition \ref{P5}, Proposition \ref{P6} and linearity,
the required solution may be written in the form
$$y(x)=\widehat{\Pi}(x)+\Phi (x,\widetilde{\Pi})-\Theta (x,\Pi )-\frac{1}{n}\Upsilon (x)+\frac{3}{n}\Psi (x)+(\Pi (0)-\frac{3}{n})x.$$

\medskip

(3) It remains to consider the case $2n+m\neq\pm n$. Equation (\ref{eq53}) now takes the form
$$H(x)+\frac{n}{2n+m}H(x-1)=\frac{1}{2n+m},$$
and we immediately obtain
$$H(x)=L_{b}(e_{-n/(2n+m)}(x-b))+\frac{1}{3n+m},$$
We next proceed, by using Proposition \ref{P2}, to solving
$$g(x)-g(x-1)=-H(x)=L_{b}(e_{-n/(2n+m)}(x-b))-\frac{1}{3n+m}$$
to obtain
$$g(x-1)=\Pi (x)+L_{b}(e_{-n/(2n+m)}(x-b))-\frac{1}{3n+m}x.$$
Solving next $y(x+1)$ from
$$y(x+2)-y(x+1)=-g(x-1)=\Pi (x)+L_{b}(e_{-n/(2n+m)}(x-b))+\frac{1}{3n+m}x,$$
we conclude that
$$y(x+1)=\widetilde{\Pi} (x)+\Phi (x,\Pi)+\Pi (0)x+\frac{1}{3n+m}(\Psi (x)-x)+L_{b}(e_{-n/(2n+m)}(x-b)).$$
Shifting now $x+1$ to $x$, recalling the difference equations satisfied by $\Psi (x)$ and $e_{-n/m+2n}(x-c_{j})$, and observing that $\Phi (x-1,\Pi )$ can be replaced as $\Phi (x,\Pi )$ as in the preceding case, we get
$$y(x)=\widetilde{\Pi}(x)+\Phi (x,\Pi)+(\Pi (0)-\frac{2}{3n+m})x+\frac{1}{3n+m}\Psi (x)+L_{b}(e_{-n/(2n+m)}(x-b)).$$
\end{proof}

\begin{theorem}\label{re2} If $n+m+p+q=0$ and $3n+2m+p\neq 0$, then the following conclusions hold for tropical meromorphic solutions $y(x)$ of equation (\ref{eq5}):

(1)  If $n+m=0$ and $n=p$, then $y(x)=\Xi (x)+\Pi (x)-\frac{1}{2n}x$.

(2) If $n+m=0$ and $n\neq\pm p$, then $y(x)=L_{b}(e_{-n/p}(\frac{x}{2}-b))+\Pi (x)-\frac{1}{n+p}x$.

(3) If $n+m\neq0$, $\frac{n+m}{n+m+p}=2$ and $\frac{n}{n+m+p}=1$, then

$$y(x)=\Pi (x)+\widetilde{\Xi}(x)-\frac{1}{2}[x-x_{0}]\Xi (x)-\frac{1}{4}x,$$
where $\Xi (x_{0})=0$.

(4) If $n+m\neq 0$ and $\left(\frac{n+m}{n+m+p}\right)^{2}-\frac{4n}{n+m+p}=0$ and $\frac{n}{n+m+p}\neq1$, then $$y(x)=\Pi (x)-\frac{1}{4}x+L_{b}(e_{\frac{n+m}{2(n+m+p)}}(x-b)$$

and if $n+m\neq 0$ and $\left(\frac{n+m}{n+m+p}\right)^{2}-\frac{4n}{n+m+p}>0$, then
$$y(x)=\Pi (x)-\frac{1}{4}x+L_{b}(e_{\alpha}(x-b))+L_{b}(e_{\beta}(x-b)),$$
where $\alpha ,\beta$ are the roots of $\lambda^{2}+\frac{n+m}{n+m+p}\lambda+\frac{n}{n+m+p}=0$.
\end{theorem}

\begin{remark}\label{re2rem} Observe that by Remark \ref{ThmBrem}, the case when $n+m+p+q=0$, $3n+2m+p\neq 0$, $n+m\neq 0$ and $\left(\frac{n+m}{n+m+p}\right)^{2}-\frac{4n}{n+m+p}<0$ remains open.
\end{remark}

\begin{proof} Suppose first that we have, in addition, $n+m=0$. Then $n+m+p=p\neq 0$. If now $n=-p$, we have $3n+2m+p=2(n+m)=0$, contradicting our assumption. If then $n=p$, then one may immediately see that $m=-n$ and $q=-n$, and (\ref{eq5}) takes the form

\begin{equation}\label{eq51deg}
y(x)-y(x+1)+y(x+2)-y(x+3)=\frac{1}{n}.
\end{equation}

Setting $G(x):=y(x)+y(x+2)$,  $(\ref{eq51deg})$ turns into $G(x+1)-G(x)=-\frac{1}{n}.$
Thus, $G(x)=\Pi (x)-\frac{1}{n}x$, hence $y(x+2)+y(x)=\Pi (x)-\frac{1}{n}x.$ It is now easy to solve this equation termwise to obtain $y(x)=\Xi (x)+\frac{\Pi (x)}{2}-\frac{1}{2n}x+\frac{1}{2n}$. The claim now follows by changing the notation. It remains to assume that $n\neq\pm p$.  Equation now $(\ref{eq5})$ reduces to

\begin{equation}\label{eq51deg33}
ny(x)-ny(x+1)+py(x+2)-py(x+3)=1.
\end{equation}

By setting $F(x):=ny(x)+py(x+2)$, we get $F(x+1)-F(x)=-1$. Thus $$F(x)=py(x+2)+ny(x)=\Pi (x)-x.$$ By an easy modification of Theorem \ref{ThmA} and of Proposition \ref{P12}, we have $y(x)=L_{b}(e_{-n/p}(\frac{x}{2}-b))+\frac{\Pi (x)}{n+p}-\frac{x}{n+p}+\frac{2p}{(n+p)^{2}},$ where $x_{1},\ldots ,x_{K}$ are the slope discontinuities of $f(x)$ in the interval $[0,1)$. The claim again follows by a slight change of notation.

For the rest of this proof, we have $n+m+p\neq0$ and  $n+m\neq0$. Thus $(\ref{eq51})$ may be written as

\begin{equation}\label{eq5123}
g(x+1)+\frac{n+m}{n+m+p}g(x)+\frac{n}{n+m+p}g(x-1)=1.
\end{equation}

To apply now Theorem \ref{ThmB}, we have to look all cases therein separately.

Case $\frac{n+m}{n+m+p}=-2$ and $\frac{n}{n+m+p}=1$. In this case, $3n+m=0$ and $m+p=0$. This implies that $3n+2m+p=0$, a contradiction.

Case $\frac{n+m}{n+m+p}=2$ and $\frac{n}{n+m+p}=1$. In this case, we have $m=-p$, $m=n$ and $q=-n$. Equation now $(\ref{eq5})$ reduces to

\begin{equation}\label{eq51deg34}
ny(x)-py(x+1)+py(x+2)-ny(x+3)=1.
\end{equation}

Denoting, as before, $g(x):=y(x+2)-y(x+3)$, and further, $H(x):=g(x)-\frac{1}{4}$, we obtain
$$H(x+1)+2H(x)+H(x-1)=0.$$
Recalling now Theorem \ref{ThmB}, we have
$$g(x)=y(x+2)-y(x+3)=\widetilde{\Xi}(x)+[x-x_{0}]\Xi (x)+\frac{1}{4},$$
where $\Xi (x_{0})=0$. Therefore, $f(x)=-y(x+2)$ satisfies
$$f(x+1)-f(x)=\widetilde{\Xi}(x)+[x-x_{0}]\Xi (x)+\frac{1}{4}.$$
By Proposition \ref{P3} and Proposition \ref{P31}, we conclude that
$$f(x)=\Pi (x)-\frac{1}{2}\widetilde{\Xi}(x)-\frac{1}{2}[x-x_{0}]\Xi (x)+\frac{1}{4}\Xi (x)+\frac{1}{4}x.$$
Therefore, by combining together $1$-periodic functions, resp. $2$-periodic, anti-$1$-periodic functions into just one such respective function, by slight change in notation and simplifying, we finally obtain
$$y(x)=\Pi (x)+\widetilde{\Xi}(x)-\frac{1}{2}[x-x_{0}]\Xi (x)-\frac{1}{4}x,$$
where $\Xi (x_{0})=0$.

Case $\left(\frac{n+m}{n+m+p}\right)^{2}-\frac{4n}{n+m+p}=0$ and $\frac{n}{n+m+p}\neq1$, we may use the preceding notations to obtain

\begin{equation}\label{eq61}
H(x+1)+\frac{n+m}{n+m+p}H(x)+\frac{n}{n+m+p}H(x-1)=0.
\end{equation}

Therefore, all tropical meromorphic functions of type
$$H(x)=y(x+2)-y(x+3)=L_{b}(e_{\frac{n+m}{2(n+m+p)}}(x-b))$$
are solutions to (\ref{eq61}). Denoting again $f(x):=-y(x+2)$, we get
$$f(x+1)-f(x)=L_{b}(e_{\frac{n+m}{2(n+m+p)}}(x-b))+\frac{1}{4},$$
hence
$$f(x)=\Pi (x)+\frac{1}{4}x+L_{b}(e_{\frac{n+m}{2(n+m+p)}}(x-b))$$
and so
$$y(x)=\Pi (x)-\frac{1}{4}x+L_{b}(e_{\frac{n+m}{2(n+m+p)}}(x-b)).$$

Case $\left(\frac{n+m}{n+m+p}\right)^{2}-\frac{4n}{n+m+p}>0$, we have
$$H(x)=y(x+2)-y(x+3)=L_{b}(e_{\alpha}(x-b))+L_{b}(e_{\beta}(x-b))+\frac{1}{4},$$
where $\alpha ,\beta$ are the roots of $\lambda^{2}+\frac{n+m}{n+m+p}\lambda+\frac{n}{n+m+p}=0$. This implies that
$$f(x)=-y(x+2)=\Pi (x)+\frac{1}{4}x+L_{b}(e_{\alpha}(x-b))+L_{b}(e_{\beta}(x-b)).$$
Therefore,
$$y(x)=\Pi (x)-\frac{1}{4}x+L_{b}(e_{\alpha}(x-b))+L_{b}(e_{\beta}(x-b)).$$
\end{proof}

\begin{remark} Observe that in Case $\left(\frac{n+m}{n+m+p}\right)^{2}-\frac{4n}{n+m+p}=0$ and $\frac{n}{n+m+p}\neq1$, all solutions to (\ref{eq61}) are
$$H(x)=L_{b}(e_{\frac{n+m}{2(n+m+p)}}(x-b))+[x-b_{0}]L_{b}(e_{\frac{n+m}{2(n+m+p)}}(x-b))+\frac{1}{4},$$
where $e_{\frac{n+m}{2(n+m+p)}}(b_{0}-b)=0$. To find all solutions $y(x)$ to $y(x+2)-y(x+3)=H(x)$, we should be able to find tropical meromorphic solutions to $f(x+1)-f(x)=[x-b_{0}]e_{\frac{n+m}{2(n+m+p)}}(x-b)$. This remains open to us.
\end{remark}

\subsection{The case $n+m+p+q\neq 0$.} In this case, equation (\ref{eq5}) may be written as
\begin{equation}\label{eq6}
F(x+2)+\frac{p}{q}F(x+1)+\frac{m}{q}F(x)+\frac{n}{q}F(x-1)=0
\end{equation}
by setting $F(x)=y(x+1)-y(x)-\frac{c}{n+m+p+q}$. We will discuss all possible forms of $F(x)$ in the following, then the expressions of $y(x)$ follows easily.

Let $\xi_{1}, \xi_{2},\xi_{3}$ be the roots of $\lambda^{3}+\frac{p}{q}\lambda^{2}+\frac{m}{q}\lambda+\frac{n}{q}=0;$ to avoid similar complications as in Theorem \ref{ThmB}, we assume that all roots $\xi_{1}, \xi_{2},\xi_{3}$ are real. Vieta's formulas imply that
\begin{equation}
\left\{
 \begin{array}{r@{\;\;}l}
&\xi_{1}\xi_{2}\xi_{3}=-\frac{n}{q},\\
& \xi_{1}+\xi_{2}+\xi_{3}=-\frac{p}{q},\\
 &\xi_{1}\xi_{2}+\xi_{2}\xi_{3}+\xi_{3}\xi_{1}=\frac{m}{q},\\
 \end{array}
 \right.
 \end{equation}
where $\xi_{1}, \xi_{2},\xi_{3}$ are not equal to $0,1$.
Let
\begin{eqnarray}\label{eqlp}
G(x)&=&F(x+1)-(\xi_{1}+\xi_{2})F(x)+\xi_{1}\xi_{2}F(x-1)\nonumber\\
&=&F(x+1)-\xi_{1}F(x)-\xi_{2}[F(x)-\xi_{1}F(x-1)].
\end{eqnarray}
Then the equation $(\ref{eq6})$ equals to
\begin{equation}\label{eq7}
G(x+1)=\xi_{3}G(x).
\end{equation}
The considerations are now divided in two cases, $\xi_{3}\neq-1$ and $\xi_{3}=-1$.

{\bf Case 1.} If $\xi_{3}\neq-1$, denote first $H(x):=F(x)-\xi_{1}F(x-1)$. From $(\ref{eqlp})$ and $(\ref{eq7})$, we have
\begin{equation}\label{eq78}
H(x+1)-\xi_{2}H(x)=G(x)=L_{b}(e_{\xi_{3}}(x-b)).
\end{equation}
The necessary reasoning for this case is now to be carried through in a number of subcases, depending on possible relatiions between the roots $\xi_{1},\xi_{2},\xi_{3}$.

{\bf Case 2.} If $\xi_{3}=-1$, then from $(\ref{eqlp})$ and $(\ref{eq7})$, we now obtain
\begin{equation}\label{eq7811}
H(x+1)-\xi_{2}H(x)=G(x)=\Xi (x).
\end{equation}
Again, a number of subcases are to be treated separately.

\section{Hayman conjecture in tropical setting}\label{Hconj}

In this section, we proceed to presenting tropical versions on the Hayman conjecture \cite{hayman1},  recalled as follows:

{\noindent\bf Hayman conjecture.} Let $f$ be a transcendental meromorphic function and $n\geq1$. Then
$f^{n}f'-1$ has infinitely many zeros.

Actually, Hayman proved the claim for $n\geq 3$, and Mues \cite{mues} for $n=2$. The final case $n=1$ has been proved, later on, by Clunie \cite{clunie} for transcendental entire functions, Bergweiler and Eremenko \cite{BE}, Chen and Fang \cite{chenfang} for transcendental meromorphic functions.

Laine and Yang \cite{LY}, Theorem 2, proposed a  difference analogue to the Hayman conjecture, proving

{\noindent\bf Theorem A.} Let $f$ be a transcendental entire
function of finite order and $a$ be a non-zero constant. If
$n\geq2$, then $f(z)^nf(z+c)-a$ has infinitely many zeros.

Liu and Yang \cite{LiuYang}, Theorem 1.4, also proved a related result on the value distribution of difference polynomials.

{\noindent\bf Theorem B.} Let $f$ be a transcendental entire
function of finite order and $a$ be a non-zero constant. If
$n\geq2$, then $f(z)^n[f(z+c)-f(z)]-a$ has infinitely many zeros.

We now consider tropical versions of the Hayman conjecture. In fact, we consider values of $f(x)^{\otimes \alpha}\otimes f(x+c)$ for different $\alpha$. This problem can be expressed as the problem of the existence of tropical meromorphic solutions of ultra-discrete equations of type $f(x)^{\otimes \alpha}\otimes f(x+c)=b(x)$. We also can consider equations of type $f(x)^{\otimes \alpha}\otimes [f(x+c)\oslash f(x)]=b(x)$, i.e. of equations of type $f(x)^{\otimes \alpha-1}\otimes f(x+c)=b(x)$.

As our first observation, we have
\begin{lemma}\label{ent} If $f(x)$ is a non-linear tropical entire function and $\alpha >0$, then $f(x)^{\otimes \alpha}\otimes f(x+c)$ is tropical entire as well.
\end{lemma}

\begin{proof} Suppose $G(x):=f(x)^{\otimes \alpha}\otimes f(x+c)=\alpha f(x)+f(x+c)$ has a pole at $x_{0}$, say. If $f(x)$ has no root at $x_{0}$, then $f$ must have a pole at $x_{0}+c$, contradicting the fact that $f$ has no poles. If then $\alpha f(x)$ has a root at $x_{0}$, then $-\alpha f(x)$ has a pole at $x_{0}$, hence $f(x+c)=G(x)-\alpha f(x)$ has a pole at $x_{0}$, a contradiction again.
\end{proof}

\begin{theorem}\label{th5}
If $f(x)$ is a non-linear tropical entire function and $\alpha>0$, then $f(x)^{\otimes \alpha}\otimes f(x+c)$ must have at least one root.
\end{theorem}

\begin{proof}
By the preceding lemma, $G(x):=f(x)^{\otimes \alpha}\otimes f(x+c)=\alpha f(x)+f(x+c)$ is tropical entire. Assume that $f(x)^{\otimes \alpha}\otimes f(x+c)$ has no roots. Then $G(x)$ should be a linear function $px+q$, thus
\begin{equation}\label{lala}
\alpha f(x)+f(x+c)=px+q,
\end{equation}
where $p$ and $q$ are constants. Since $f(x)$ is non-linear tropical entire function, which implies that $\alpha f(x)$ has at least one root, say at $x_{0}$. But then $x_{0}$ is a pole of $f(x+c)$, a contradiction.
\end{proof}

We now state a partial tropical counterpart to the Hayman conjecture:

\begin{theorem}\label{th56}
If $f(x)$ is a tropical transcendental entire function and if $\alpha>0$, then $f(x)^{\otimes \alpha}\otimes f(x+c)$ must have infinitely many roots.
\end{theorem}

\begin{proof}
By Lemma \ref{ent}, $G(x):=f(x)^{\otimes \alpha}\otimes f(x+c)$ is tropical entire. If it has, contrary to the assertion, finitely many roots only, then $G(x)$ is a tropical polynomial. Let $x_{1},\ldots x_{n}$ be its roots. If $\alpha f(x)$ has a root at $x$ such that $x_{j}<x<x_{j+1}$, then $G(x)$ is linear around $x$, hence $f$ must have a pole at $x+j$, a contradiction. Therefore, the only possible roots of $f$ are at $\{ x_{1},\ldots ,x_{n},x_{1}+c,\ldots ,x_{n}+c\}$, implying that $f$ is a tropical polynomial, contradicting the assumption that $f$ is transcendental.
\end{proof}

\begin{remark} If $\alpha<0$, the conclusion of Theorem \ref{th56} is not true. For example, if $\alpha=-2$, then the tropical exponential function $e_{2}(x)$ satisfies $e_{2}(x+1)-2e_{2}(x)=0$. This implies that $e_{2}(x)^{\otimes(-2)}\otimes e_{2}(x+1)=e_{2}(x+1)-2e_{2}(x)$ has no roots.
\end{remark}

\begin{theorem}\label{th57}
If $f(x)$ is a tropical transcendental meromorphic function with hyper-order $\rho_{2}(f)<1$ and $\alpha\neq\pm1$, then $f(x)^{\otimes \alpha}\otimes f(x+c)$ cannot be a linear function.
\end{theorem}

\begin{proof}
Let $f(x)$ be a tropical meromorphic function. Assume that
\begin{equation}\label{haha}
f(x)^{\otimes \alpha}\otimes f(x+c)=ax+b,
\end{equation}
where $a,b$ are constants. Obviously, $f_{0}(x)=\frac{a}{1+\alpha}x+\frac{b}{1+\alpha}-\frac{ac}{(1+\alpha)^{2}}$ is a solution of (\ref{haha}). All tropical meromorphic functions of (\ref{haha}) can be represented as the sum of $f_{0}(x)$ and the tropical meromorphic solutions of $g(x)^{\otimes \alpha}\otimes g(x+c)=0$. By \cite[Theorem 9.1]{Laine1}, $g(x)^{\otimes \alpha}\otimes g(x+c)=0$ has no non-constant tropical meromorphic solutions of hyper-order $\rho_{2}(f)<1$. But then $f=f_{0}+g$ is a tropical polynomial, contradicting the assumption that $f$ is transcendental.
\end{proof}

\begin{remark} The following two examples to show that if $\alpha=\pm1$, Theorem \ref{th57} is not true.
\end{remark}

\begin{example}\label{expp} If $\alpha=1$, Theorem $\ref{th57}$ is not true. This can be seen by the following example, due to Chen \footnote{L. Chen, \emph{Tropical meromorphic functions and their application on difference equations}, National Chengchi University, Master's thesis.}. Define $\pi_{a}(x)=\max\{(1-a)([x]-x), a([-x]-(-x))\}$ for $0\leq a <1$, see the graphs below for $a=0.5$. From these graphs, we observe that $\pi_{a}(x)+\pi_{a}(x-0.5)\equiv -1$ is a constant. By an easy calculation, we see that the order of $\rho(\pi_{a}(x))=2$, hence the hyper-order $\rho_{2}(\pi_{a}(x))=0$.

\begin{figure}[h]
\begin{minipage}[t]{1\linewidth}
\centering
\includegraphics[width=0.5\textwidth]{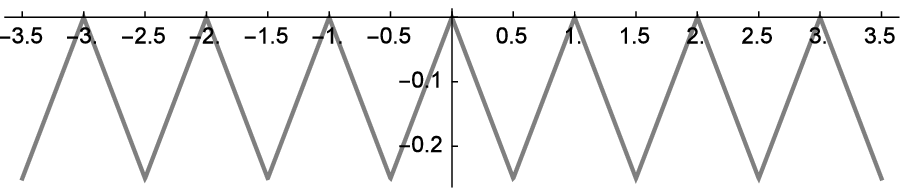}
\caption{$\pi_{a}(x),a=0.5$}\label{good}
\end{minipage}
\end{figure}

\begin{figure}[h]
\begin{minipage}[t]{1\linewidth}
\centering
\includegraphics[width=0.5\textwidth]{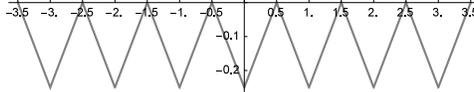}
\caption{$\pi_{a}(x-0.5),a=0.5$}\label{gooddd}
\end{minipage}
\end{figure}
\end{example}

\begin{example} The counter-example when $\alpha=-1$ can be from $\Psi(x)$ in Proposition $\ref{P4}$ with modifications.
Define
\begin{equation}
F(x)=\left\{
 \begin{array}{r@{\;,\;}l}
\sum_{j=0}^{[\frac{x}{q}]}\max(0,x-qj) &x\geq q,\\
\sum_{j=0}^{[\frac{x}{q}]+1}\max(0,x-qj) & 0\leq x<q,\\
 x+\sum_{j=[\frac{x}{q}]}^{0}\max(0,-x+jq) & x<0.\\
 \end{array}
 \right.
\end{equation}
We see that $F(x)$ is a continuous piecewise linear function, hence $F(x)$ is a tropical meromorphic function of $\rho_{2}(f)=0$ and satisfies $F(x)-F(x-q)=x$, which implies that $F(x)^{\otimes(-1)}\otimes F(x-q)=-x.$
\end{example}

\section{Br\"{u}ck conjecture in the tropical setting}\label{Bruck}

In this final section, we add some remarks on the tropical counterpart to the Br\"{u}ck conjecture \cite{bruck}, originally formulated as follows:

{\noindent\bf Conjecture.} If $f$ is a non-constant entire function with hyper-order $\rho_{2}(f)<+\infty$, where $\rho_{2}(f)$ is not a positive integer, and if $f$ and $f'$ share one finite value $a$ CM, then $$\frac{f'(z)-a}{f(z)-a}=c$$ for some constant $c\neq0$.

Recently, using difference analogues of logarithmic derivative lemma, Heittokangas $et$ $al$  \cite{janne} proved a difference result related to the Br\"{u}ck conjecture:

{\noindent\bf Theorem C\cite[Theorem 2.1]{janne}}. Let $f$ be a transcendental meromorphic function with finite order $\rho(f)<2$, and let $c\in\mathbf{C}$. If $f(z)$ and $f(z+c)$ share the value $a\in\mathbf{C}$ and $\infty$ CM, then
\begin{equation}
\frac{f(z+c)-a}{f(z)-a}=\tau \nonumber
\end{equation}
for some constants $\tau$.

As a tropical counterpart to the preceding Theorem C, we prove

\begin{theorem}\label{tb}
Let $f(x)$ be a tropical entire function, $a\in\mathbf{R}$ be fixed and suppose that $\max (f(x+1),a)$ and $\max (f(x),a)$ have the same roots with the same multiplicity for all $x\in\mathbf{R}$. Then three alternatives may appear, namely either
$$f(x)=\Pi (x)+Bx$$
for all $x$ close enough to $-\infty$, or
$$f(x)=\Pi (x)+Bx$$
for all $x$ close enough to $+\infty$, or
$$f(x)=A\Psi (x)+\Pi (x)+(B-A)x$$
for all $x$ such that $|x|$ is large enough. Here $A,B$ are real constants, $\Pi (x)$ is a tropical $1$-periodic function and $\Psi (x)$ is as defined in (\ref{psi2}).
\end{theorem}

\begin{proof} By \cite{KLT}, Proposition 6.5, $(\max (f(x+1),a))\oslash (\max (f(x),a))$ is linear, hence
\begin{equation}\label{eqn}
\max (f(x+1),a)=\max (f(x),a)+Ax+B
\end{equation}
for some real constants $A,B$. As $f(x)$ is tropical entire, it cannot be upper bounded.

Suppose first that either $\lim_{x\rightarrow -\infty}f(x)\in\mathbf{R}\cup\{ -\infty\}$ or $\lim_{x\rightarrow +\infty}f(x)\in\mathbf{R}\cup\{ -\infty\}$. If we now have $x\rightarrow +\infty$, then it is immediate to conclude from (\ref{eqn}) that $A=0$. Letting then $x$ be close enough to $-\infty$, (\ref{eqn}) takes the form
$$f(x+1)=f(x)+B.$$
Solving this implies $f(x)=\Pi (x)+Bx$. Similar reasoning applies the case of $\lim_{x\rightarrow -\infty}f(x)\in\mathbf{R}\cup\{ -\infty\}$.

It remains to consider the case when $|f(x)|\rightarrow +\infty$ as $x\rightarrow\pm\infty$. Suppose first $x\rightarrow +\infty$. Letting $x$ be large enough so that (\ref{eqn}) takes the form
$$f(x+1)=f(x)+Ax+B.$$
An immediate solution to this equation is $f(x)=A\Psi (x)+\Pi (x)+(B-A)x$ for all $x$ large enough. The corresponding case when $x\rightarrow -\infty$ is immediate.
\end{proof}

\begin{remark} The corresponding considerations for $f(x)$ being tropical meromorphic and non-entire are apparently more complicated, to be treated elsewhere.
\end{remark}



\begin{thebibliography}{}

\bibitem{BE} W. Bergweiler and A. Eremenko, \emph{On the singularities of the inverse to a meromorphic function
of finite order, Revista Matem\'{a}tica Iberoamericana.} \textbf{11}, 355--373, (1995).

\bibitem{bruck} R. Br\"{u}ck, \emph{On entire functions which share one
value CM with their first derivative,} Results
Math. \textbf{30}, 21--24, (1996).

\bibitem{chenfang} H. H. Chen and M. L. Fang , \emph{On the value distribution
of $f^{n}f'$}, Sci. China Ser. A. \textbf{38}, 789--798, (1995).

\bibitem{clunie} J. Clunie, \emph{On a result of Hayman}, J. London. Math. Soc. \textbf{42}, 389--392, (1967).

\bibitem{ggg1} G. Gundersen, \emph{Meromorphic solutions of} $f^{6}+g^{6}+h^{6}\equiv 1$, Analysis. \textbf{18}, 285--290, (1998).

\bibitem{ggg2} G. Gundersen, \emph{Meromorphic solutions of} $f^{5}+g^{5}+h^{5}\equiv 1$, Complex Var. \textbf{43}, 293--298, (2001).

\bibitem{halburd} R. G. Halburd and N. Southall, \emph{Tropical Nevanlinna theory on ultra-discrete equations,} Int. math. Res. Notices. \textbf{2009}, 887--911, (2009).

\bibitem{hayman1} W. K. Hayman, \emph{Picard values of meromorphic functions
and their derivatives,} Ann. Math. \textbf{70}, 9--42, (1959).

\bibitem{hayman2} W. K. Hayman, \emph{Waring's Problem f\"{u}r analytische Funktionen}, Bayer. Akad. Wiss. Math.-Natur. Kl. Sitzungsber, 1--13, (1985).

\bibitem{janne} J. Heittokangas, R. Korhonen, I. Laine, J. Rieppo
and J. L. Zhang, \emph{Value sharing results for shifts
of meromorphic functions, and sufficient conditions
for periodicity,} J. Math. Anal. Appl.
\textbf{355}, no. 1, 352--363, (2009).

\bibitem{KLT} R. Korhonen, I. Laine and K. Tohge, \emph{Tropical Value Distribution Theory and Ultra-discrete Equations}, World Scientific, Singapore, 2015.

\bibitem{Laine1} I. Laine and K. Tohge, \emph{Tropical Nevanlinna theory and second main theorem}, Proc. Lond. Math. Soc. \textbf{102}), 883--922, (2011.

\bibitem{LY} I. Laine and C. C. Yang, \emph{Value distribution of difference polynomials}, Proc. Japan Acad. Ser. A \textbf{83}, 148--151, (2007).

\bibitem{LiuYang} K. Liu and L. Z. Yang, \emph{Value distribution of the difference operator}, Arch. Math. \textbf{92}, 270--278, (2009).

\bibitem{mues} E. Mues, \emph{\"{U}ber ein Problem von Hayman}, Mathematische Zeitschrift. \textbf{164}, 3, 239--259,(1979).

\bibitem{DB} D. Speyer and B. Sturmfels, \emph{Tropical Mathematics,} Math. Mag. \textbf{82}, 163--173, (2009).

\bibitem{wiles2} R. Taylor and A. Wiles, \emph{Ring-theoretic properties of certain Hecke algebras}, Ann. Math. \textbf{141}, 553--572, (1995).

\bibitem{Tohge} K. Tohge, \emph{The order and type formulas for tropical entire functions-another flexibility of complex analysis}, In: Proceedings of the Workshop on Complex Analysis and its Applications to Differential and Functional Equations, Joensuu, 113--164, (2014).

\bibitem{wiles1} A. Wiles, \emph{Modular elliptic curves and Fermat's last theorem,} Ann. Math. \textbf{141}, 443--551, (1995).

\end{thebibliography}


\end{document}